\documentclass[12pt]{amsart}  

\usepackage{custom_preamble} 

\begin{document}

\title{On a theorem of Shkredov}

\author{Tom Sanders}
\address{Department of Pure Mathematics and Mathematical Statistics\\
University of Cambridge\\
Wilberforce Road\\
Cambridge CB3 0WB\\
England } \email{t.sanders@dpmms.cam.ac.uk}

\begin{abstract}
We show that if $A$ is a finite subset of an abelian group with additive energy at least $c|A|^3$ then there is a set $\mathcal{L} \subset A$ with $|\mathcal{L}| = O(c^{-1}\log |A|)$ such that $|A \cap \Span(\mathcal{L})| = \Omega(c^{1/3}|A|)$.
\end{abstract}

\maketitle

\section{Introduction and notation}

We shall prove the following theorem which is a slight strengthening of
\cite[Theorem 1.5]{shk::3}.
\begin{theorem}\label{thm.shkredovtheorem}
Suppose that $G$ is an abelian group and $A \subset G$ is a finite
set with $\|1_A \ast 1_A\|_{\ell^2(G)}^2 \geq c|A|^3$. Then there
is a set $\mathcal{L} \subset A$ with $|\mathcal{L}| =O(c^{-1} \log
|A|)$ such that\footnote{Recall that $\Span(\mathcal{L})$ is the set of all sums $\sum_{x \in \mathcal{L}}{\sigma_x.x}$ where $\sigma \in \{-1,0,1\}^\mathcal{L}$.} $|A \cap \Span (\mathcal{L}) | =\Omega( c^{1/3}|A|)$.
\end{theorem}
It is immediate from the Cauchy-Schwarz inequality that if $|A+A| \leq K|A|$ then $\|1_A \ast 1_A\|_{\ell^2(G)}^2 \geq |A|^3/K$ whence the conclusion of the above result applies to $A$. This was noted by Shkredov in \cite[Corollary 3.2]{shk::3}, however, something slightly stronger is also true.\footnote{Since writing this note it has come to the author's attention (personal communication) that Shkredov has also independently proved Theorem \ref{thm.sumsetshkredov}.}
\begin{theorem}\label{thm.sumsetshkredov}
Suppose that $G$ is an abelian group and $A \subset G$ is a finite
set with $|A+A| \leq K|A|$. Then there is a set $\mathcal{L} \subset
A$ with $|\mathcal{L}| =O(K \log |A|)$ such that $A \subset \Span (\mathcal{L})$.
\end{theorem}
Before we begin with our proofs it will be useful to recall some well-known tools; Rudin \cite{rud::1} is the classic reference for these.

A subset $\mathcal{L}$ of an abelian group $G$ is said to be \emph{dissociated} if
\begin{equation*}
\sum_{x \in \mathcal{L}}{\sigma_x.x} = 0_G \textrm{ and } \sigma \in \{-1,0,1\}^\mathcal{L}
\textrm{ implies that } \sigma \equiv 0.
\end{equation*}
Algebraically, dissociativity is particularly useful in view of the following easy lemma.
\begin{lemma}\label{lem.maximaldissociated}
Suppose that $G$ is an abelian group and $A \subset G$ is finite. If
$\mathcal{L} \subset A$ is a maximal dissociated subset of $A$ then
$A \subset \Span (\mathcal{L})$.
\end{lemma}
Analytically, dissociativity can be handled very effectively using the Fourier transform which we take a moment to introduce.

Suppose that $G$ is a (discrete) abelian group. We write $\wh{G}$
for the dual group, that is the compact abelian group of
homomorphisms from $G$ to $S^1:=\{z \in \C: |z|=1\}$ endowed with the Haar probability measure $\mu_{\wh{G}}$, and define the Fourier transform of a function $f \in \ell^1(G)$ to be
\begin{equation*}
\wh{f}:\wh{G} \rightarrow \C; \gamma \mapsto \sum_{x \in
G}{f(x)\overline{\gamma(x)}}.
\end{equation*}
The following result is a key tool in harmonic analysis.
\begin{proposition}[Rudin's inequality]
Suppose that $G$ is an abelian group and $\mathcal{L}
\subset G$ is a dissociated set. Then, for each $p \in [2,\infty)$
we have
\begin{equation*}
\|\wh{f}\|_{L^p(\mu_{\wh{G}})} =O(\sqrt{p}
\|f\|_{\ell^2(\mathcal{L})}) \textrm{ for all } f \in
\ell^2(\mathcal{L}).
\end{equation*}
\end{proposition}
The proof may be found in many places (e.g. \cite{rud::1}) and proceeds for even integral values of $p$ (from which the general result follows immediately) where one may apply Parseval's theorem to get a physical space expression which counts additive relations; dissociativity tells us that there are few of these and so the norm is small.

\section{The proof of Theorem \ref{thm.shkredovtheorem}}

Our proof of Theorem \ref{thm.shkredovtheorem} is guided by Shkredov \cite{shk::3} although we are able to make some  simplifications and improvements by using some standard facts about the $L^p(\mu_{\wh{G}})$-norms.  

We require the following lemma which is implicit in the paper \cite{bou::4} of Bourgain.

\begin{lemma}\label{lem.bourgain}
Suppose that $G$ is a abelian group, $A \subset G$ is finite, $l$ is a positive integer and $p \geq 2$. Then there is a set $A' \subset A$ such that all dissociated subsets of $A'$ have size at most $l$ and
\begin{equation*}
\|\wh{1_A}-\wh{1_{A'}}\|_{L^p(\mu_{\wh{G}})} =  O(\sqrt{p/l}|A|).
\end{equation*}
\end{lemma}
\begin{proof}
We define sets $A_0 \supset A_1 \supset \dots \supset A_s$ and $\mathcal{L}_0,\mathcal{L}_1,\dots,\mathcal{L}_s$ iteratively starting with $A_0:=A$. Suppose that we have defined $A_i$.
\begin{enumerate}
\item If there is no dissociated subset of $A_i$ with size $l$ then
terminate the iteration;
\item if there is a dissociated subset of $A_i$ with size $l$ then let
$\mathcal{L}_i$ be any such set and put $A_{i+1} = A_i \setminus
\mathcal{L}_i$.
\end{enumerate}
The algorithm terminates at some stage $s$ with $s \leq |A|/l$ since $|A_{i+1}| =|A_i| -l$. Write $A':=A_s$ which consequently has no dissociated subset of size greater than $l$.

Since $A$ is the disjoint union of the sets $\mathcal{L}_0,\dots,\mathcal{L}_{s-1}$ and $A'$, and the Fourier transform is linear we have
\begin{equation*}
\|\wh{1_A}-\wh{1_{A'}}\|_{L^p(\mu_{\wh{G}})} = \|\sum_{i=0}^{s-1}{\wh{1_{\mathcal{L}_i}}}\|_{L^p(\mu_{\wh{G}})} \leq \sum_{i=0}^{s-1}{\|\wh{1_{\mathcal{L}_i}}\|_{L^p(\mu_{\wh{G}})}}.
\end{equation*}
Now each summand is $O(\sqrt{p}\|1_{\mathcal{L}_i}\|_{\ell^2(\mathcal{L}_i)}) = O(\sqrt{pl})$, by Rudin's inequality, whence
\begin{equation*}
\|\wh{1_A}-\wh{1_{A'}}\|_{L^p(\mu_{\wh{G}})} = O(s\sqrt{pl}) = O(\sqrt{p/l}|A|),
\end{equation*}
in view of the upper bound on $s$.
\end{proof}

\begin{proof}[Proof of Theorem \ref{thm.shkredovtheorem}]
Write $p:=2+\log |A|$ and let $l$ be an integer with
\begin{equation*}
l=O(p c^{-(p-2)/p}|A|^{2/p}) = O(c^{-1} \log |A|)
\end{equation*}
such that when we apply Lemma \ref{lem.bourgain} to $A$ we get a set $A' \subset A$ for which
\begin{equation}\label{eqn.lemmaerror}
\|\wh{1_A} - \wh{1_{A'}}\|_{L^p(\mu_{\wh{G}})} \leq c^{(p-2)/2p}|A|^{(p-1)/p}/4.
\end{equation}
Let $\mathcal{L}$ be a maximal dissociated subset of $A'$. We have $|\mathcal{L}| \leq l= O(c^{-1}\log |A|)$ by the choice of $l$, and $A' \subset \Span (\mathcal{L})$ by Lemma \ref{lem.maximaldissociated}, whence $|A \cap \Span (\mathcal{L})| \geq |A'|$ and the result will follow from a lower bound on $|A'|$.

By the $\log$-convexity of the $L^p(\mu_{\wh{G}})$ norms we have
\begin{eqnarray*}
\|\wh{1_A}-\wh{1_{A'}}\|_{L^4(\mu_{\wh{G}})}^4 & \leq &
\|\wh{1_A}-\wh{1_{A'}}\|_{L^2(\mu_{\wh{G}})}^{(2p-8)/(p-2)}\|\wh{1_A}-\wh{1_{A'}}\|_{L^p(\mu_{\wh{G}})}^{2p/(p-2)}\\ & = & \|1_{A \setminus A'}\|_{L^2(\mu_{\wh{G}})}^{(2p-8)/(p-2)}\|\wh{1_A}-\wh{1_{A'}}\|_{L^p(\mu_{\wh{G}})}^{2p/(p-2)}\\& \leq &  |A|^{(p-4)/(p-2)}\|\wh{1_A}-\wh{1_{A'}}\|_{L^p(\mu_{\wh{G}})}^{2p/(p-2)},
\end{eqnarray*}
by Parseval's theorem and the fact that $A' \subset A$. Now, inserting the bound in (\ref{eqn.lemmaerror}) we get that
\begin{equation*}
\|\wh{1_A}-\wh{1_{A'}}\|_{L^4(\mu_{\wh{G}})}^4 \leq |A|^{(p-4)/(p-2)}. c |A|^{(2p-4)/(p-2)}/2^{4p/(p-2)} \leq \|\wh{1_A}\|_{L^4(\mu_{\wh{G}})}^4/2^4.
\end{equation*}
On the other hand, by the triangle inequality,
\begin{equation*}
\|\wh{1_A}-\wh{1_{A'}}\|_{L^4(\mu_{\wh{G}})} \geq \|\wh{1_A}\|_{L^4(\mu_{\wh{G}})}- \|\wh{1_{A'}}\|_{L^4(\mu_{\wh{G}})}
\end{equation*}
whence, on combination with the previous, we have
\begin{equation}\label{eqn.lwr}
\|1_{A'}\|_{L^4(\mu_{\wh{G}})}^4 \geq \|1_A\|_{L^4(\mu_{\wh{G}})}^4/2^4 \geq c|A|^3/2^4.
\end{equation}
Finally, we note that
\begin{equation*}
\|\wh{1_{A'}}\|_{L^4(\mu_{\wh{G}})}^4 \leq \|\wh{1_{A'}}\|_{L^2(\mu_{\wh{G}})}^2\|\wh{1_{A'}}\|_{L^\infty(\mu_{\wh{G}})}^2 \leq |A'|^3,
\end{equation*}
by H{\"o}lder's inequality, Parseval's theorem and the Hausdorff-Young inequality.  The result follows on taking cube roots.
\end{proof}

\section{The proof of Theorem \ref{thm.sumsetshkredov}}

The proof is essentially Theorem 6.10 of L{\'o}pez and Ross
\cite{lopros::} coupled with Lemma \ref{lem.maximaldissociated}.
\begin{proof}[Proof of Theorem \ref{thm.sumsetshkredov}]
Write $f:=1_{A+A}\ast 1_{-A}$. Then
\begin{eqnarray*}
\|\wh{f}\|_{L^1(\mu_{\wh{G}})} & = &
\int{|\wh{1_{A+A}}(\gamma)\wh{1_{-A}}(\gamma)|d\mu_{\wh{G}}(\gamma)}\\
& \leq &
\left(\int{|\wh{1_{A+A}}(\gamma)|^2d\mu_{\wh{G}}(\gamma)}\right)^{1/2}\left(\int{|\wh{1_{-A}}(\gamma)|^2d\mu_{\wh{G}}(\gamma)}\right)^{1/2}\\
& = & \sqrt{|A+A||-A|}\leq \sqrt{K}|A|,
\end{eqnarray*}
by the Cauchy-Schwarz inequality, Parseval's theorem and the
doubling condition $|A+A| \leq K|A|$. Furthermore
$\|f\|_{\ell^\infty(G)} \leq |A|$ and $\|f\|_{\ell^1(G)} = |A||A+A|$
and so
\begin{equation*}
\|\wh{f}\|_{L^2(\mu_{\wh{G}})}^2 = \|f\|_{\ell^2(G)}^2 \leq
\|f\|_{\ell^\infty(G)}.\|f\|_{\ell^1(G)} \leq |A|^2|A+A| \leq
K|A|^3,
\end{equation*}
by Parseval's theorem, H\"{o}lder's inequality and the doubling
condition. Whence, by $\log$-convexity of the $L^{p'}(\mu_{\wh{G}})$
norms, we have
\begin{equation*}
\|\wh{f}\|_{L^{p'}(\mu_{\wh{G}})} \leq \sqrt{K}|A|^{2-1/p'}
\textrm{ for all } p' \in [1,2].
\end{equation*}

Suppose that $\mathcal{L}$ is a maximal dissociated subset of $A$
and $(p,p')$ is a conjugate pair of exponents with $p' \in (1,2]$.
Then, by Rudin's inequality, we have
\begin{eqnarray*}
\|f\|_{\ell^2(\mathcal{L})}^2 = \langle
\wh{f1_\mathcal{L}},\wh{f}\rangle_{L^2(\mu_{\wh{G}})} & \leq &
\|\wh{f1_{\mathcal{L}}}\|_{L^p(\mu_{\wh{G}})}\|\wh{f}\|_{L^{p'}(\mu_{\wh{G}})}\\
&= & O(
\sqrt{p}.\|f\|_{\ell^2(\mathcal{L})}\|\wh{f}\|_{L^{p'}(\mu_{\wh{G}})}).
\end{eqnarray*}
The construction of $f$ ensures that for any $a \in A$, $f(a) =
1_{A+A} \ast 1_{-A}(a) \geq |A|$ so, canceling
$\|f\|_{\ell^2(\mathcal{L})}$ above we get
\begin{equation*}
\sqrt{|\mathcal{L}|.|A|^2} \leq \|f\|_{\ell^2(\mathcal{L})} =
O(\sqrt{p}\|\wh{f}\|_{L^{p'}(\mu_{\wh{G}})}) =
O(\sqrt{pK}|A|^{2-1/p'}).
\end{equation*}
Putting $p=2+\log |A|$ and some rearrangement tells us that
$|\mathcal{L}| =O(K \log |A|)$. Since $\mathcal{L}$ was maximal
Lemma \ref{lem.maximaldissociated} then yields the result.
\end{proof}

\section{Concluding remarks}

In some ways the results are close to best possible.  In Theorem \ref{thm.sumsetshkredov} suppose that $A$ is the union of $K$ highly dissociated points and a long arithmetic progression (or subgroup if $G$ has a lot of torsion).  It is easy to see that in this case if $\mathcal{L}$ is such that $A \subset \Span(\mathcal{L})$ then $|\mathcal{L}| = \Omega(K + \log |A|)$.  If $K$ is around $\log^{O(1)}|A|$ then this is close to the upper bound in the theorem; if $K=\log^{o(1)}|A|$ then there are better results known: this is the celebrated Green-Ruzsa-Fre{\u\i}man theorem \cite{greruz::0}.  

In Theorem \ref{thm.shkredovtheorem} the $c^{1/3}$ cannot be improved: consider an arithmetic progression (or, again, subgroup if $G$ has a lot of torsion) of length $c^{1/3}|A|$ unioned with $|A|$ dissociated points.  This satisfies the lower bound on the energy but one cannot hope to find any structure other than the progression, \emph{i.e.} in more than a proportion $\Omega(c^{1/3})$ of the set.  Of course, this bound is not the important bound in the result; the bound on $\mathcal{L}$ is what is really of interest.

\bibliographystyle{alpha}

\bibliography{references}

\end{document}